\documentclass[preprint,a4paper]{elsarticle}

\usepackage[T1]{fontenc}
\usepackage[utf8]{inputenc}
\usepackage[english]{babel}
\usepackage{libertinus}
\usepackage{microtype}
\usepackage{mathtools, amsfonts, amssymb, amsthm}
\usepackage{enumitem}
\setlist{nosep}
\usepackage{xcolor}
\usepackage{graphicx}
\usepackage{doi}
\usepackage{hyperref}
\hypersetup{
  colorlinks=true,
  linkcolor=blue,
  filecolor=magenta,
  urlcolor=cyan,
  citecolor=blue
}
\usepackage{bookmark}
\usepackage[nameinlink,noabbrev]{cleveref}
\usepackage{thmtools} 

\declaretheorem[name=Theorem,numberwithin=section]{theorem}

\declaretheorem[name=Proposition,sibling=theorem]{proposition}
\declaretheorem[name=Corollary,sibling=theorem]{corollary}

\declaretheorem[style=definition,name=Remark,sibling=theorem]{remark}
\declaretheorem[style=definition,name=Example,sibling=theorem]{example}

\begin{document}

\begin{frontmatter}

\title{From curvature to Kovacic: a geometric approach to integrability of scalar ODEs}

\author[uca]{A.~J. Pan-Collantes}
\ead{antonio.pan@uca.es}

\author[iesj]{J.~A. \'{A}lvarez-Garc\'{\i}a}
\ead{jose.alvg@gmail.com}

\address[uca]{Departamento de Matem\'aticas, Universidad de C\'adiz -- UCA, Puerto Real, Spain}
\address[iesj]{Departamento de Matem\'aticas, IES Jorge Juan, Junta de Andaluc\'\i{}a, San Fernando, Spain}

\begin{abstract}
We study first-order ordinary differential equations such that the intrinsic Gauss curvature of the associated surface depends only on the independent variable: $\mathcal{K}(x,u)=\kappa(x)$, showing that this geometrically motivated class of equations admits a threefold connection to the second-order linear operator $L=d^2/dx^2+\kappa(x)$: the divergence along every solution satisfies a Riccati equation that linearizes to $L(y)=0$; every solution of the first-order equation satisfies the non-homogeneous equation $L(u)=c(x)$; and solutions of $L(y)=0$ give rise to integrating factors for the original nonlinear equation. By means of differential Galois theory, we prove that the nonlinear equation is integrable by quadratures if and only if $L$ admits a non-zero Liouvillian solution; when $\kappa$ is rational, Kovacic's algorithm provides a complete decision procedure.
\end{abstract}

\begin{keyword}
geometric integrability \sep intrinsic curvature \sep Riccati equation \sep Schr\"{o}dinger operator \sep differential Galois theory \sep Kovacic's algorithm \sep Liouvillian solutions
\MSC[2020]{34A05, 34B24, 34C14, 12H05, 53A35}
\end{keyword}

\end{frontmatter}

\section{Introduction}\label{sec:intro}

The integration of scalar first-order ordinary differential equations (ODEs) 
\begin{equation}\label{eq:ode1_intro}
u'(x) = \phi(x, u)
\end{equation}
is a fundamental problem. Since Lie's foundational work, the dominant approach has been symmetry analysis: if the equation admits a one-parameter group of point symmetries, it can be reduced to quadratures \cite{Olver1993, BlumanAnco2002, Stephani1989}. Nevertheless, the search for symmetries remains a nontrivial task for a general equation, and complementary routes to integrability are of considerable interest.

The present paper develops one such route for a geometrically distinguished class of equations, based on the intrinsic curvature of a Riemannian surface associated with the equation. The idea of attaching geometric structures to differential equations has a long history, from Cartan's equivalence problem \cite{Olver1995} to Arnold's geometric methods \cite{ArnoldODE1992}. In the framework introduced in \cite{Bayrakdar2018a,surfaces2025,integrationJacobifields}, a Riemannian metric is defined on the $(x,u)$-plane so that the solutions of $u' = \phi(x,u)$ correspond to a distinguished family of geodesics. The Gauss curvature of this surface, $\mathcal K$, encodes integrability information. 

In prior work \cite{surfaces2025, integrationJacobifields}, the authors showed that when $\mathcal{K}$ is constant the equation is integrable by quadratures. The present paper addresses the next natural generalization: the class of equations for which $\mathcal{K}(x,u) = \kappa(x)$, i.e., the curvature depends on the independent variable alone.

It was already shown in \cite{integrationJacobifields} that this curvature condition allows for the construction of integrating factors from solutions $y$ of $L(y)=0$, where $L$ is the operator
\begin{equation}\label{eq:schrodinger_operator}
L=\frac{d^2}{dx^2}+\kappa(x).
\end{equation}
Our first result (\Cref{thm:riccati_connection}) establishes that the curvature condition $\mathcal{K}(x,u) = \kappa(x)$ holds if and only if the divergence of the vector field associated with the ODE, evaluated along every solution, satisfies the Riccati equation $p'+p^2+\kappa(x)=0$. Through the classical substitution $p=y'/y$ \cite{Ince1956, Reid1972}, the Schr\"odinger equation $y''+\kappa(x)y=0$ emerges, providing another connection of the original nonlinear ODE to the linear operator $L$ in \eqref{eq:schrodinger_operator}.

The link to this linear operator runs deeper. We prove (\Cref{prop:nonhomog}) that every solution $u$ of the first-order ODE \eqref{eq:ode1_intro} satisfying $\mathcal{K}(x,u)=\kappa(x)$ is also a solution of the non-homogeneous linear equation 
\begin{equation}\label{eq:nonhomog_intro}
L(u)=c
\end{equation}
where $c=c(x)$ is determined by $\phi$ but independent of $u$. Moreover, this embedding characterizes the class: a first-order ODE belongs to the relevant class if and only if its solution set $\Gamma$ is contained in the two-dimensional affine space $S$ generated by the kernel of $L$ and a fixed particular solution (\Cref{prop:nonhomog}). 

The inclusion $\Gamma\subset S$ has further consequences for integrability. Using differential Galois theory \cite{vanderPutSinger2003, Singer1981}, we prove (\Cref{thm:galois_integrability}) that the first-order ODE is integrable by quadratures whenever the operator $L$ admits a non-zero Liouvillian solution. When $\kappa$ is rational, Kovacic's algorithm \cite{Kovacic1986} provides an effective decision procedure, making Liouvillian integrability algorithmically decidable for this class of nonlinear equations.
This is complementary to the Prelle--Singer procedure \cite{PrelleSinger1983} and its extensions to first-order ODEs with Liouvillian solutions \cite{DuarteDuarteMota2002a,DuarteDuarteMota2002b}, and to Prelle--Singer-based methods for linearization of nonlinear ODEs \cite{ChandrasekharEtAl2005}: rather than searching directly in the nonlinear equation, our approach exploits the reduction to the Galois theory of~$L$.

The paper is organized as follows. \Cref{preliminaries} reviews the geometric framework of Riemannian surfaces associated with first-order ODEs and recalls the curvature formula. \Cref{section:riccati} establishes the equivalence between the curvature condition $\mathcal{K}(x,u)=\kappa(x)$ and the Riccati dynamics of the flow divergence along solutions (\Cref{thm:riccati_connection}). \Cref{sec:embedding} proves that every solution of the nonlinear ODE satisfies a fixed non-homogeneous second-order linear equation (\Cref{prop:nonhomog}), and characterizes the class as precisely those first-order ODEs whose solution set is contained in the affine space $S$ (\Cref{prop:nonhomog}). \Cref{sec:jacobi} recalls the integrating-factor construction and connects it to the Riccati--Schrödinger link. \Cref{geometric_interpretation} provides a projective interpretation: the Riccati solutions arising from \Cref{thm:riccati_connection} are identified with tangent directions to the solution locus $\Gamma$ inside the affine plane $S$, yielding a projective analogue of the Gauss map (\Cref{prop:tangent_interpretation}). Finally, \Cref{sec:differential_galois} develops the differential Galois consequences: \Cref{thm:galois_integrability} identifies Liouvillian integrability of the nonlinear ODE with Liouvillian solvability of $L$, and shows that, when $\kappa\in\mathbb{C}(x)$, Kovacic's algorithm provides an effective and complete decision procedure.

\section{Geometric framework}\label{preliminaries}
Consider a scalar first-order ODE of the form:
\begin{equation} \label{eq:ode1}
    u'(x) = \phi(x, u)
\end{equation}
where $u$ is the dependent variable, $x$ is the independent variable, and $\phi(x, u)$ is a sufficiently smooth function.

Throughout, we assume that $\phi$ is defined on an open set $D\subset\mathbb{R}^2$ and sufficiently smooth. All assertions are understood locally: we may replace $D$ and the intervals of definition of solutions by suitable smaller open sets/intervals ensuring that every quantity appearing is defined and sufficiently smooth, without further comment. We will also assume implicitly that $D$ is such that for each $x\in \mathbb R$ the set
\[
D_x:=\{u\in\mathbb{R}:(x,u)\in D\}
\]
is connected (e.g., an interval, possibly empty or unbounded).

To simplify notation, we will often omit the explicit dependence on $(x,u)$ in expressions involving $\phi$ and its partial derivatives when no confusion can arise.

Equation \eqref{eq:ode1} is geometrically encoded by the vector field $A$ on the $(x,u)$ plane:
\begin{equation} \label{eq:vector_field}
    A = \partial_x + \phi\partial_u,
\end{equation}
being the integral curves of this vector field in correspondence with its solutions. In recent years, a geometric framework has been developed to study first-order ODEs through the association of a surface, in the sense of a 2-dimensional Riemannian manifold \cite{Bayrakdar2018a,integrationJacobifields,surfaces2025}, such that $A$ is a geodesic vector field with respect to the induced metric.

In this framework, the metric tensor $g$ associated with the first-order ODE is given by:
\begin{align}\label{metricBay}
    g&=( 1+ \phi^2)dx\otimes dx-\phi dx\otimes du-\phi du\otimes dx+du\otimes du,
\end{align}
and the corresponding volume form is
\begin{equation}\label{volume_form}
    \Omega = \sqrt{\det(g)} dx \wedge du = dx \wedge du.
\end{equation}

The Gauss curvature $\mathcal K$ associated with the surface is given by the expression
\begin{equation} \label{eq:curvature_formula}
    \mathcal{K}(x,u) = -\partial_u(A(\phi)),
\end{equation}
providing significant insights into the behavior of the equation \cite{Bayrakdar2018a,surfaces2025,integrationJacobifields}.

In \cite{surfaces2025,integrationJacobifields} the notion of relative Jacobi field was introduced, and used to show that if the curvature $\mathcal{K}$ is constant, then the ODE \eqref{eq:ode1} is integrable by quadratures. More generally, in \cite{integrationJacobifields} the following result is established for the case where the curvature depends only on the independent variable:

\begin{theorem}[\cite{surfaces2025,integrationJacobifields}]\label{thm:jacobi_integration}
Let $u' = \phi(x,u)$ be a first-order ODE whose curvature satisfies $\mathcal{K}(x,u) = \kappa(x)$, and consider the associated Schrödinger-type equation
\begin{equation} \label{eq:linear_2ode_intro}
    y''(x) + \kappa(x)y(x) = 0.
\end{equation}
If $\delta$ is a non-zero solution of \eqref{eq:linear_2ode_intro}, then $\delta$ determines a relative Jacobi field from which an integrating factor for the original nonlinear equation \eqref{eq:ode1} can be constructed. In particular, a single non-zero solution of \eqref{eq:linear_2ode_intro} suffices to obtain the general solution of \eqref{eq:ode1} by quadratures.
\end{theorem}

As is well-known, equation \eqref{eq:linear_2ode_intro} has a strong relation to the Riccati equation 
\begin{equation}\label{Riccatieq}
p'(x) + p(x)^2 + \kappa(x) = 0
\end{equation}
via the standard substitution $p = y'/y$; see, e.g., \cite{Teschl2012}. In the following section, we will show a more direct link of the curvature condition $\mathcal{K}(x,u) = \kappa(x)$ to the Riccati equation \eqref{Riccatieq}.

\section{Riccati dynamics along solutions}\label{section:riccati}

In this section, we show that the curvature depends exclusively on the independent variable precisely when the divergence of $A$ satisfies the Riccati equation \eqref{Riccatieq}, which in turn establishes a direct link to operator \eqref{eq:schrodinger_operator}.

Observe that the divergence of the vector field $A$ defined in \eqref{eq:vector_field} with respect to the volume form $\Omega$ given in \eqref{volume_form}
is given by:
\begin{equation} \label{eq:divergence}
    \text{div}\,A =  \phi_u. 
\end{equation}

On the other hand, given a solution $f(x)$ of the ODE \eqref{eq:ode1}, we can evaluate the divergence along the solution curve $(x,f(x))$, which provides a function of the independent variable $x$ alone. This function captures the local expansion or contraction of the flow generated by $A$ along the solution trajectory, and we will denote it by:
\begin{equation} \label{eq:p_definition}
    p_f(x) := \phi_u(x, f(x)).
\end{equation}

Observe that this function depends on the choice of the particular solution $f$; when the underlying solution $f$ is fixed in the discussion, we will suppress the subscript and simply write $p(x)$ for simplicity, except when the dependence on $f$ needs to be emphasized.

The following theorem establishes a fundamental equivalence between the curvature condition $\mathcal{K}(x,u) = \kappa(x)$ and the behavior of the functions $p_f(x)$.

\begin{theorem}\label{thm:riccati_connection}
The curvature corresponding to equation \eqref{eq:ode1} satisfies $\mathcal{K}(x,u)=\kappa(x)$ on $D$ if and only if, for every solution $f$, the function $p_f(x):=\phi_u(x,f(x))$ satisfies the Riccati equation \eqref{Riccatieq} on its interval of definition.
\end{theorem}

\begin{proof}
Let $f$ be any solution of \eqref{eq:ode1}. By the chain rule and $f'(x)=\phi(x,f(x))$,
\[
p_f'(x)=\frac{d}{dx}\phi_u(x,f(x))=\phi_{xu}(x,f(x))+\phi(x,f(x))\,\phi_{uu}(x,f(x))=A(\phi_u)(x,f(x)),
\]
where $A$ is the associated vector field \eqref{eq:vector_field}.

By equation \eqref{eq:curvature_formula} we have
\[
\mathcal{K}(x,u)=-\partial_u\bigl(A(\phi)\bigr)(x,u)=-(\phi_{xu}+\phi_u^{2}+\phi\,\phi_{uu})(x,u)=-A(\phi_u)(x,u)-\phi_u^2(x,u). 
\]
Evaluating at $u=f(x)$ and using the previous identities yields, along the solution curve,
\begin{equation}\label{eq:riccati_identity}
p_f'(x)+p_f(x)^2+\mathcal{K}(x,f(x))=0.
\end{equation}
If $\mathcal{K}(x,u)=\kappa(x)$, this reduces to the Riccati equation \eqref{Riccatieq} for $p_f$, for every solution $f$.

Conversely, assume that there exists a function $\kappa=\kappa(x)$ such that for every solution $f$ the corresponding $p_f$ satisfies \eqref{Riccatieq}. Comparing with the identity \eqref{eq:riccati_identity} gives $\mathcal{K}(x,f(x))=\kappa(x)$ along each solution curve. Given any point $(x_0,u_0)$ in the domain of $\phi$, local existence for the initial value problem $u'=\phi(x,u)$, $u(x_0)=u_0$, provides a solution $f$ with $f(x_0)=u_0$, hence
\[
\mathcal{K}(x_0,u_0)=\mathcal{K}(x_0,f(x_0))=\kappa(x_0).
\]
Since $(x_0,u_0)$ is arbitrary, it follows that $\mathcal{K}(x,u)=\kappa(x)$ on the domain.
\end{proof}

\begin{remark}
Theorem \ref{thm:riccati_connection} establishes a mapping from solutions $f(x)$ of the first-order ODE \eqref{eq:ode1} to solutions $p_f(x)$ of the Riccati equation \eqref{Riccatieq}. However, this mapping is not injective in general. In fact, the function $p_f$ is the same for all solutions $f$ if $\phi$ is affine with respect to its second argument, i.e., $\phi_{uu} = 0$.
\end{remark}

\begin{example}[Linear equations]\label{ex:linear}
Let $\phi(x,u)=B(x)u+C(x)$. Since $\phi_{uu}=0$, we have
\[
\phi_x+\phi\,\phi_u = \bigl(B'(x)+B(x)^2\bigr)u+\bigl(C'(x)+B(x)C(x)\bigr),
\]
whence $\mathcal{K}(x,u)=-\bigl(B'(x)+B(x)^2\bigr)=:\kappa(x)$. The divergence $p_f(x)=\phi_u(x,f(x))=B(x)$ is identical for every solution~$f$, so the map of \Cref{thm:riccati_connection} collapses to a single Riccati solution, illustrating the situation described in the preceding remark.

As a concrete instance, take $B(x)=-\tan x$ and $C(x)=0$, so that $u'=-\tan(x)\,u$. Then $\kappa(x)=-(-\sec^2 x+\tan^2 x)=1$, and the associated Schr\"odinger equation is $y''+y=0$.
\end{example}

\begin{example}[A nonlinear Euler--Cauchy flow]\label{ex:euler_cauchy}
Consider the nonlinear ODE defined for $x>0$ in the region $u^2>4x$:
\begin{equation}\label{eq:euler_cauchy_ode}
u' = \frac{u}{2x}+\frac{3}{2x}\sqrt{u^2-4x}.
\end{equation}
A direct computation yields $\phi_x+\phi\,\phi_u = 2u/x^2$, whence $\mathcal{K}(x,u) = -2/x^2 = :\kappa(x)$. 

On the other hand, the ODE admits the one-parameter family of solutions
\begin{equation}\label{eq:euler_cauchy_sol}
u(x;s) = s x^2 + s^{-1} x^{-1}, \qquad s \neq 0,
\end{equation}
as can be verified by substitution. For each solution, the divergence $p_s(x) = \phi_u(x, u(x;s))$ is given by
\[
p_s(x) = \frac{2s^2 x^3 + 1}{x(s^2 x^3 - 1)},
\]
and one checks that $p_s$ satisfies $p_s' + p_s^2 - 2/x^2 = 0$. In contrast to the linear case of \Cref{ex:linear}, the divergence now depends on the chosen particular solution so, as established in \Cref{thm:riccati_connection}, the map $s \mapsto p_s$ produces a genuinely one-parameter family of solutions to the Riccati equation \eqref{Riccatieq}. The geometric aspects of this example will be developed in later sections.
\end{example}

To close this section, we note that the Riccati equation \eqref{Riccatieq} for the flow divergence is linearized by the standard substitution $p(x)=y'(x)/y(x)$, which yields the Schr\"odinger-type equation \eqref{eq:linear_2ode_intro}. Consequently, in addition to the integrating-factor approach of \cite{surfaces2025,integrationJacobifields}, we obtain here a further link between the class of equations satisfying \(\mathcal{K}(x,u)=\kappa(x)\) and the operator $L=\frac{d^{2}}{dx^{2}}+\kappa(x)$.

\section{Second-order linear embedding}\label{sec:embedding}

In this section we investigate a further connection between the class of first-order ODEs satisfying $\mathcal{K}(x,u)=\kappa(x)$ and the associated linear operator $L=\frac{d^2}{dx^2}+\kappa(x)$. The following proposition gives a complete characterization of the class in terms of an affine embedding of the solution set.

\begin{proposition}\label{prop:nonhomog}
Given the ODE \eqref{eq:ode1}, we have that $\mathcal{K}(x,u)=\kappa(x)$ on $D$ if and only if there exists a continuous function $c\colon I\to\mathbb{R}$ such that every solution of \eqref{eq:ode1} with graph in $D$ satisfies
\begin{equation}\label{eq:nonhomog}
    L(u) = c(x)
\end{equation}
wherever defined. Equivalently, the solution set of \eqref{eq:ode1} is contained in the affine space $S := u_p + V$, where $V := \ker L$ and $u_p$ is any particular solution of \eqref{eq:nonhomog}.
\end{proposition}

\begin{proof}
Define $C(x,u):=\phi_x+\phi\,\phi_u+\kappa(x)u$ on $D$. If $u$ is a solution of \eqref{eq:ode1}, differentiating $u'=\phi(x,u)$ gives
\begin{equation}\label{eq:second_order_identity_Ca}
u''(x)+\kappa(x)u(x)=C\bigl(x,u(x)\bigr).
\end{equation}

Suppose first that $\mathcal{K}(x,u)=\kappa(x)$. Then
\[
\partial_u C(x,u)=\phi_{xu}+\phi_u^2+\phi\,\phi_{uu}+\kappa(x)=-\mathcal{K}(x,u)+\kappa(x)=0,
\]
so $C$ is independent of $u$ on each connected fiber $D_x$. Setting $c(x):=C(x,u)$ for any $(x,u)\in D$, equation \eqref{eq:second_order_identity_Ca} becomes $L(u)=c(x)$.

Conversely, suppose every solution satisfies $L(u)=c(x)$. Substituting into \eqref{eq:second_order_identity_Ca} gives $C(x,u(x))=c(x)$ along every solution. By local existence, every $(x_0,u_0)\in D$ lies on some solution, so $C(x,u)\equiv c(x)$ on all of $D$. Differentiating in $u$ yields
\[
0=\partial_u C(x,u)=-\mathcal{K}(x,u)+\kappa(x),
\]
hence $\mathcal{K}(x,u)=\kappa(x)$ on $D$.

The affine-space reformulation follows because the solution set of $L(u)=c(x)$ is exactly $u_p+V=S$.
\end{proof}

\begin{corollary}\label{cor:one_dim_locus}
If $\mathcal{K}(x,u)=\kappa(x)$, the solution set $\Gamma$ of \eqref{eq:ode1} is a smooth curve inside the two-dimensional affine space $S$. Every $u\in\Gamma$ has the form
\begin{equation}\label{eq:affine_sol}
  u(x)=C_1 y_1(x)+C_2 y_2(x)+u_p(x),
\end{equation}
where $(C_1,C_2)$ are constrained to a one-dimensional locus determined by $\phi$.
\end{corollary}
\begin{proof}
By \Cref{prop:nonhomog}, $\Gamma\subset S$, so every $u\in\Gamma$ has the form \eqref{eq:affine_sol} for unique constants $(C_1,C_2)$. Fix $x_0$; evaluating \eqref{eq:affine_sol} and its derivative at $x_0$ gives
\[
  C_1 y_1(x_0)+C_2 y_2(x_0) = u_0 - u_p(x_0), \qquad
  C_1 y_1'(x_0)+C_2 y_2'(x_0) = \phi(x_0,u_0) - u_p'(x_0),
\]
where $u_0 := u(x_0)$ and we used $u'(x_0)=\phi(x_0,u_0)$. Since the coefficient matrix of this $2\times 2$ linear system is the Wronskian $W(x_0)\neq 0$, the pair $(C_1,C_2)$ is uniquely and smoothly determined by $u_0\in D_{x_0}$. Thus the assignment $u_0\mapsto(C_1(u_0),C_2(u_0))$ is a smooth injective map from the one-dimensional fiber $D_{x_0}$, and $\Gamma$ is its image — a smooth curve in the parameter plane.
\end{proof}

\Cref{prop:nonhomog} thus characterizes the class $\mathcal{K}=\kappa(x)$ by the affine embedding $\Gamma\subset S$; \Cref{cor:one_dim_locus} makes precise that $\Gamma$ is, within this two-dimensional affine space, a smooth curve of codimension one.

We can also show that every smooth curve in $S$ arises in this way.

\begin{proposition}\label{prop:converse_embedding}
Given a smooth immersed curve $\Gamma\subset S$, there exists, locally, a first-order ODE $u'=\phi(x,u)$ with $\mathcal{K}(x,u)=\kappa(x)$ whose solution set is precisely~$\Gamma$.
\end{proposition}

\begin{proof}
Let $(y_1,y_2)$ be a fundamental system for $L(y)=0$, and $u_p$ a particular solution of $L(u)=c$. Parametrizing $\Gamma$ by $s\mapsto(c_1(s),c_2(s))$, each point of the curve corresponds to the function
\[
f(x;s):= c_1(s)\,y_1(x)+c_2(s)\,y_2(x)+u_p(x)\in S.
\]
The partial derivative $\partial_s f(x;s)=c_1'(s)\,y_1(x)+c_2'(s)\,y_2(x)$ is a non-trivial element of $V=\ker L$, since $(c_1'(s),c_2'(s))\neq(0,0)$ (the curve is immersed) and $y_1,y_2$ are linearly independent. In particular $\partial_s f(x;s)\neq 0$ for all $x$ outside the discrete zero set of this element of~$V$. By the implicit function theorem, for each $(x_0,s_0)$ with $\partial_s f(x_0;s_0)\neq 0$, the map $s\mapsto f(x_0;s)$ is a local diffeomorphism, so $s$ can be expressed locally as a smooth function $s=\sigma(x,u)$ satisfying $f(x;\sigma(x,u))=u$. Setting
\[
\phi(x,u):= f_x(x;\sigma(x,u)),
\]
we obtain a smooth function on a domain $D\subset I\times\mathbb{R}$. By construction, $f(\cdot\,;s)$ solves $u'=\phi(x,u)$ for each~$s$, and every solution satisfies $L(u)=c(x)$ since $f(\cdot\,;s)\in S$. By the converse implication in \Cref{prop:nonhomog}, $\mathcal{K}(x,u)=\kappa(x)$ on~$D$.
\end{proof}

Together with \Cref{prop:nonhomog} and \Cref{cor:one_dim_locus}, \Cref{prop:converse_embedding} establishes, once $\kappa$ and $c$ are fixed, a local bijection between smooth immersed curves in the affine plane $S$ and first-order ODEs with curvature $\mathcal{K}(x,u)=\kappa(x)$, in such a way that the curve~$\Gamma$ encodes the nonlinearity~$\phi$.

\begin{example}[Continuing with \Cref{ex:linear}]\label{ex:linear_embedding}
For the equation 
$$
u'=-\tan(x)\,u,
$$we have $\kappa=1$ and the inhomogeneous coefficient $C(x)=0$. A direct computation shows that $c(x)=C'(x)+B(x)C(x)=0$, so \eqref{eq:nonhomog} reduces to the homogeneous equation
\[
u''+u=0,
\]
with fundamental system $\{y_1,y_2\}=\{\cos x,\sin x\}$ and particular solution $u_p=0$. The general solution of the first-order ODE is $u(x;s)=s\cos x$, which lies in $S=V=\ker L$ as expected. In the representation \eqref{eq:affine_sol} (with $u_p=0$), this corresponds to $(C_1,C_2)=(s,0)$: the locus $\Gamma$ is the $C_1$-axis in the parameter plane, a straight line. This reflects the fact that for a linear ODE ($\phi_{uu}=0$) the constraint in \Cref{cor:one_dim_locus} is itself linear.
\end{example}

\begin{example}[Continuing with \Cref{ex:euler_cauchy}]\label{ex:euler_cauchy_embedding}
By \Cref{prop:nonhomog}, every solution of \eqref{eq:euler_cauchy_ode} satisfies the Euler--Cauchy equation
\begin{equation}\label{eq:euler_cauchy_2nd}
u''(x)-\frac{2}{x^2}\,u(x) = 0.
\end{equation}
The indicial roots are $r=2$ and $r=-1$, yielding the fundamental system $\{y_1,y_2\}=\{x^2,\,x^{-1}\}$ with Wronskian $W=-3$. Since the non-homogeneous term vanishes ($c(x)=0$), we have $u_p=0$ and $S=V$. The one-parameter family of solutions of \eqref{eq:euler_cauchy_ode} is given by \eqref{eq:euler_cauchy_sol}, as one verifies by direct substitution. In the $(C_1,C_2)$-plane, the locus $\Gamma$ is the hyperbola $C_1 C_2=1$, a genuine nonlinear curve, in contrast with the linear case of \Cref{ex:linear}.
\end{example}

\begin{example}[Two curves in the same affine space]\label{ex:complete}
Consider the following two ODEs on $x>0$: the nonlinear equation
\begin{equation}\label{eq:hyperbolic_ode}
  u' = \frac{u}{2x} + \frac{5}{2}x^2 + \frac{3}{2x}\sqrt{(u-x^3)^2 - 4x},
\end{equation}
and the linear equation
\begin{equation}\label{eq:linear_locus_ode}
  u' = \frac{2x^3 - 1}{x(x^3 + 1)}\, u + \frac{x^5 + 4x^2}{x^3 + 1}.
\end{equation}
A straightforward computation gives $\mathcal{K}(x,u)=-2/x^{2}$ for both ODEs. 

The associated operator is $L=\frac{d^{2}}{dx^{2}}-\frac{2}{x^{2}}$, with fundamental system $\{y_1,y_2\}=\{x^{2},x^{-1}\}$. By \Cref{prop:nonhomog}, the solutions of each ODE satisfy $L(u)=c(x)$ for some function $c$ depending only on $x$. By using the definition of $C(x,u)$ in the proof of \Cref{prop:nonhomog}, one checks that $c(x)=4x$ for both \eqref{eq:hyperbolic_ode} and \eqref{eq:linear_locus_ode}. Taking $u_p=x^{3}$ as a particular solution of $L(u)=4x$, both ODEs embed into the same affine space
\[
  S \;=\; x^{3}+V, \qquad V=\operatorname{span}\{x^{2},\,x^{-1}\}.
\]

It remains to identify the curves $\Gamma$ and $\Gamma'$ inside $S$ traced by each ODE \eqref{eq:hyperbolic_ode} and \eqref{eq:linear_locus_ode}, respectively. Every solution of either ODE has the form
\[
  u(x;s) = x^{3}+C_1(s)\,x^{2}+C_2(s)\,x^{-1},
\]
and substituting into each ODE determines the constraint on $(C_1,C_2)$:
\begin{itemize}
\item \textit{Locus $\Gamma$ of \eqref{eq:hyperbolic_ode}.}
Set $u = x^3+C_1 x^2+C_2 x^{-1}$, so $u' = 3x^2+2C_1 x - C_2 x^{-2}$.
Substituting into \eqref{eq:hyperbolic_ode}, and after straightforward algebraic manipulation, we obtain the identity
\[
  \Gamma \;=\; \{(C_1,C_2)\in\mathbb{R}^2 : C_1 C_2 = 1\},
\]
a rectangular hyperbola in the parameter plane.

\item \textit{Locus $\Gamma'$ of \eqref{eq:linear_locus_ode}.}
Substituting $u,u'$ into \eqref{eq:linear_locus_ode}, and simplifying, we obtain the identity $C_1=C_2.$

Thus $\Gamma'$ is the diagonal line
\[
  \Gamma' \;=\; \{(C_1,C_2)\in\mathbb{R}^2 : C_1 = C_2\}.
\]

\end{itemize}

\medskip
The two ODEs share the same curvature $\kappa(x)=-2/x^{2}$ and the same affine space $S$. What distinguishes them is the shape of the locus inside~$S$: a hyperbola $\Gamma$ for \eqref{eq:hyperbolic_ode}, a line $\Gamma'$ for \eqref{eq:linear_locus_ode}. This illustrates the bijection of \Cref{prop:converse_embedding}: once $\kappa$ and $c$ are fixed, the affine space $S$ is determined, and the curve $\Gamma\subset S$ encodes the nonlinearity of the ODE.
\end{example}

\section{Relative Jacobi fields and integrating factors}\label{sec:jacobi}

In Section \ref{preliminaries} we mentioned that the curvature condition $\mathcal{K}(x,u) = \kappa(x)$ allows for the construction of an integrating factor from the solutions of the Schrödinger-type equation \eqref{eq:linear_2ode_intro}, as it was shown in \cite{surfaces2025,integrationJacobifields} using the notion of relative Jacobi field. Indeed, the Schrödinger-type equation \eqref{eq:linear_2ode_intro} is precisely the equation governing the relative Jacobi fields in the geometric framework of \cite{surfaces2025,integrationJacobifields}, in the case of first-order ODEs with curvature depending only on $x$. This way, the connection between the class of equations distinguished in our work and the class of operators $L=\frac{d^2}{dx^2}+\kappa(x)$ is threefold.

A classical fact about the Riccati equation is that knowledge of a single particular solution enables the construction of the general solution by quadratures. We can leverage the latter connection, together with \Cref{thm:riccati_connection}, to show that an similar property holds for the nonlinear ODE \eqref{eq:ode1}:

\begin{theorem}\label{thm:integration_from_particular}
Let $u' = \phi(x,u)$ satisfy $\mathcal{K}(x,u) = \kappa(x)$, and let $f(x)$ be a particular solution. Then the general solution can be obtained by quadratures.
\end{theorem}

\begin{proof}
Let $p_f(x) = \phi_u(x, f(x))$ be the divergence of the vector field along $f(x)$, as defined in \eqref{eq:p_definition}. By \Cref{thm:riccati_connection}, $p_f(x)$ is a particular solution of the Riccati equation
\begin{equation}
    p'(x) + p(x)^2 + \kappa(x) = 0.
\end{equation}
The standard logarithmic substitution $p(x) = \delta'(x)/\delta(x)$ yields a solution of the Schrödinger equation $\delta'' + \kappa(x)\delta = 0$, given explicitly by
\begin{equation}\label{eq:delta_integral}
    \delta_f(x) := \exp\!\left(\int_{x_0}^x \phi_u(\tau, f(\tau))\, d\tau\right).
\end{equation}
In the geometric framework of \cite{surfaces2025,integrationJacobifields}, $\delta_f$ gives rise to a relative Jacobi field, from which an integrating factor for \eqref{eq:ode1} can be constructed and the general solution obtained by quadratures (see \Cref{thm:jacobi_integration}).
\end{proof}

\begin{example}[Continuing with \Cref{ex:euler_cauchy}]\label{ex:euler_cauchy_integration}
We illustrate \Cref{thm:integration_from_particular} with the Euler--Cauchy flow \eqref{eq:euler_cauchy_ode}, taking the particular solution $f(x)=x^2+x^{-1}$ (the case $s=1$ in \eqref{eq:euler_cauchy_sol}). By \Cref{ex:euler_cauchy}, the divergence along $f$ is
\[
p_f(x)=\frac{2x^3+1}{x(x^3-1)}.
\]
Observing that $p_f(x)=\frac{d}{dx}\ln\!\left|\frac{x^3-1}{x}\right|$, the quadrature \eqref{eq:delta_integral} yields
\[
\delta_f(x)=\frac{x^3-1}{x}=x^2-\frac{1}{x},
\]
and one verifies directly that $\delta_f''-\frac{2}{x^2}\,\delta_f=0$. By \Cref{thm:jacobi_integration}, this non-zero solution of the Schrödinger equation provides an integrating factor for the nonlinear ODE \eqref{eq:euler_cauchy_ode}, and the general solution is obtained by quadratures, recovering the full solution family \eqref{eq:euler_cauchy_sol}.

\end{example}

\section{Projective interpretation of \Cref{thm:riccati_connection}}\label{geometric_interpretation}

In this section we interpret \Cref{thm:riccati_connection} within the classical framework of projective geometry on the solution space $V$. Recall that the set of solutions of the Riccati equation \eqref{Riccatieq} is parametrized by the projective line $\mathbb{P}(V)$, where $V$ denotes the solution space of $y''+\kappa(x)y=0$ \cite{Ince1956,Reid1972}. We show that the Riccati solutions $p_s$ arising from \Cref{thm:riccati_connection} correspond precisely to tangent directions of the curve $\Gamma$ in the affine plane $S$, via a projective analogue of the Gauss map.

Recall from Proposition~\ref{prop:nonhomog} that every solution of the first-order ODE \eqref{eq:ode1} satisfies the non-homogeneous equation \eqref{eq:nonhomog}, whose solution set is the two-dimensional affine space $S=u_p+V$. Here $V$ is the vector space of solutions of the homogeneous equation $y''+\kappa(x)y=0$, and $u_p$ is a fixed particular solution of \eqref{eq:nonhomog}. We fix a basis $\{y_1,y_2\}$ of $V$.

For any $u\in S$ with $u\neq u_p$, define
\begin{equation}\label{eq:projectivization}
w:=\frac{(u-u_p)'}{u-u_p}.
\end{equation}
Since $y:=u-u_p\in V\setminus\{0\}$ satisfies $y''+\kappa(x)y=0$, the substitution $w=y'/y$ shows that $w$ is a solution of the Riccati equation \eqref{Riccatieq}. Moreover, replacing $y$ by $\lambda y$ for any nonzero scalar $\lambda$ leaves $w$ unchanged, so $w$ depends only on the one-dimensional subspace $[y]\in\mathbb{P}(V)$. Conversely, every solution of \eqref{Riccatieq} arises in this way. Thus, the set of solutions of the Riccati equation is identified with the projective line $\mathbb{P}(V)\cong\mathbb{RP}^1$.

Concretely, writing $u-u_p=C_1 y_1+C_2 y_2$, we have
\begin{equation}\label{eq:w_projcoords}
w=\frac{C_1 y_1'+C_2 y_2'}{C_1 y_1+C_2 y_2},
\end{equation}
and the corresponding point in $\mathbb{P}(V)$ has homogeneous coordinates $[C_1:C_2]$.

Let $f(x;s)$, $s\in J$, be a smooth one-parameter family of solutions of the first-order ODE \eqref{eq:ode1}, where $J$ is an open interval. By Proposition~\ref{prop:nonhomog}, each $f(\cdot\,;s)\in S$, so we can write
\[
f(x;s)=u_p(x)+c_1(s)\,y_1(x)+c_2(s)\,y_2(x)
\]
for smooth scalar functions $c_1,c_2:J\to\mathbb{R}$. The map $s\mapsto(c_1(s),c_2(s))$ parametrizes $\Gamma$ as a curve in the affine coordinate plane of $S$.

The following result identifies the Riccati solutions provided by Theorem~\ref{thm:riccati_connection} within this projective framework.

\begin{proposition}\label{prop:tangent_interpretation}
For each $s\in J$, the Riccati solution $p_s(x):=\phi_u(x,f(x;s))$ provided by Theorem~\ref{thm:riccati_connection} has homogeneous coordinates $[c_1'(s):c_2'(s)]$ in $\mathbb{P}(V)$. That is, $p_s$ corresponds to the tangent direction of $\Gamma$ at the point $f(\cdot\,;s)$.
\end{proposition}

\begin{proof}
Differentiating the identity $f'(x;s)=\phi(x,f(x;s))$ with respect to the parameter $s$ gives
\[
\frac{\partial}{\partial s}f'(x;s)=\phi_u(x,f(x;s))\,\frac{\partial f}{\partial s}(x;s).
\]
Since $f(x;s)=u_p(x)+c_1(s)\,y_1(x)+c_2(s)\,y_2(x)$, the left-hand side equals $c_1'(s)\,y_1'(x)+c_2'(s)\,y_2'(x)$, while $\partial_s f=c_1'(s)\,y_1(x)+c_2'(s)\,y_2(x)$. Therefore
\[
\phi_u(x,f(x;s))=\frac{c_1'(s)\,y_1'(x)+c_2'(s)\,y_2'(x)}{c_1'(s)\,y_1(x)+c_2'(s)\,y_2(x)},
\]
which by \eqref{eq:w_projcoords} is the Riccati solution with homogeneous coordinates $[c_1'(s):c_2'(s)]$.
\end{proof}

\begin{remark}\label{rem:gauss_map}
Proposition~\ref{prop:tangent_interpretation} reveals a geometric distinction between two different maps from $\Gamma$ to $\mathbb{P}(V)$. Each solution $f(\cdot\,;s)\in \Gamma$ determines a point $(c_1(s),c_2(s))$ on the curve $\Gamma\subset S$, which via the projectivization \eqref{eq:projectivization} corresponds to the Riccati solution with coordinates $[c_1(s):c_2(s)]\in\mathbb{P}(V)$. The Riccati solution $p_s$ from Theorem~\ref{thm:riccati_connection}, however, corresponds not to this point but to the tangent direction $[c_1'(s):c_2'(s)]$. The induced map $\Gamma\to\mathbb{P}(V)$ sending each point of $\Gamma$ to its tangent direction is thus a projective analogue of the Gauss map.
\end{remark}

\begin{example}[Continuing with \Cref{ex:linear}]\label{ex:linear_proj}
In the projective framework, since $\phi_u=B(x)$ is independent of the solution, the map $f(\cdot\,;s)\mapsto p_s$ from \Cref{thm:riccati_connection} collapses to the single Riccati solution $p(x)=B(x)$. Because $\Gamma$ is a line in the $(c_1,c_2)$-plane (see \Cref{ex:linear_embedding}), all its tangent directions coincide, and the projective image of $\Gamma$ under the Gauss map of \Cref{rem:gauss_map} is a single point in $\mathbb{P}(V)$.
\end{example}

\begin{example}[A nonlinear case]\label{ex:nonlinear_proj}
Consider the nonlinear ODE
\begin{equation}\label{eq:example_sqrt}
u'=\tfrac{1}{2}\sqrt{1-u^2}.
\end{equation}
A direct computation gives $\phi_x+\phi\phi_u=-u/4$, whence $\mathcal{K}=1/4=:\kappa$. Since $c(x)=\phi_x+\phi\phi_u+\kappa\, u=0$, the non-homogeneous equation \eqref{eq:nonhomog} reduces to $y''+\frac{1}{4}y=0$, so $u_p=0$ and $S=V$.

With the basis $\{y_1,y_2\}=\{\sin(x/2),\,-\cos(x/2)\}$, the one-parameter family of solutions of \eqref{eq:example_sqrt} is
\[
f(x;s)=\sin\!\bigl(\tfrac{x-s}{2}\bigr)=\cos\!\bigl(\tfrac{s}{2}\bigr)\,y_1(x)+\sin\!\bigl(\tfrac{s}{2}\bigr)\,y_2(x),
\]
so $(c_1(s),c_2(s))=(\cos(s/2),\sin(s/2))$ and $\Gamma$ is the unit circle in the $(c_1,c_2)$-plane. Computing $p_s$ directly from $\phi_u=-u/(2\sqrt{1-u^2})$ evaluated at $u=f(x;s)=\sin((x-s)/2)$ gives
\[
p_s(x)=\frac{-\sin\!\bigl(\frac{x-s}{2}\bigr)}{2\cos\!\bigl(\frac{x-s}{2}\bigr)}=\tfrac{1}{2}\tan\!\bigl(\tfrac{s-x}{2}\bigr).
\]
One can verify that $p_s$ has homogeneous coordinates $[c_1'(s):c_2'(s)]=[-\sin(s/2):\cos(s/2)]$ in $\mathbb{P}(V)$, as predicted by \Cref{prop:tangent_interpretation}. In contrast to the linear case of \Cref{ex:linear_proj}, $\Gamma$ is a genuine nonlinear curve.
\end{example}

\begin{example}[Continuing with \Cref{ex:euler_cauchy}]\label{ex:euler_cauchy_proj}
Along the solution \eqref{eq:euler_cauchy_sol}, the divergence $p_s(x)=\phi_u(x,u(x;s))$ satisfies $p_s'+p_s^2-2/x^2=0$ by \Cref{thm:riccati_connection}. 
The tangent to $\Gamma$ at the point $(s,s^{-1})$ has direction $(c_1'(s),c_2'(s))=(1,-s^{-2})$, corresponding to $[s^2:-1]\in\mathbb{P}(V)$, in agreement with \Cref{prop:tangent_interpretation}. As $s$ varies over $(0,\infty)$, this tangent direction sweeps~$\mathbb{P}(V)$, reflecting the genuinely nonlinear character of the flow.
\end{example}

\section{Differential Galois aspects}\label{sec:differential_galois}

The embedding $\Gamma\subset S$ established in \Cref{prop:nonhomog} shows that the solutions of the nonlinear first-order ODE \eqref{eq:ode1} are built from the elements of $V$. In this section we use differential Galois theory to make the integrability consequences of this embedding precise, and we show that, when $\kappa\in\mathbb{C}(x)$, where $\mathbb{C}(x)$ denotes the field of rational functions with complex coefficients, Kovacic's algorithm \cite{Kovacic1986} provides an effective decision procedure for the Liouvillian integrability of the class of equations satisfying $\mathcal{K}(x,u)=\kappa(x)$. A key consequence is that the embedding controls the analytic complexity of the nonlinear equation from both sides.

We briefly recall the necessary background; for a comprehensive treatment see \cite{vanderPutSinger2003}. Let $F$ be a differential field with derivation $\partial$ and algebraically closed field of constants $\mathcal{C}$. A \emph{Liouvillian extension} of $F$ is a differential field extension $E\supset F$ obtained by a finite tower of intermediate extensions, each adjoining either an integral (an element $a$ with $a'\in F_i$), an exponential of an integral (an element $a$ with $a'/a\in F_i$), or an algebraic element over the preceding field $F_i$. The \emph{Picard--Vessiot extension} for a homogeneous linear ODE $L(y)=0$ over $F$ is the minimal differential field extension containing a full set of solutions; its \emph{differential Galois group} $\operatorname{Gal}(E/F)$ is a linear algebraic group over $\mathcal{C}$. A fundamental result of differential Galois theory states that $L(y)=0$ admits a non-zero Liouvillian solution if and only if the identity component $G^0$ of $\operatorname{Gal}(E/F)$ is solvable \cite{vanderPutSinger2003}. For Liouvillian first integrals of nonlinear equations, see \cite{Singer1992}. The Galoisian obstruction approach was developed in the Hamiltonian context by Morales-Ruiz and Ramis \cite{MoralesRuizRamis2001}; the present results provide an analogue for first-order nonlinear ODEs. For a treatment of differential Galois theory applied specifically to Schr\"odinger-type operators $L = \frac{d^2}{dx^2} + \kappa(x)$, see \cite{AcostaHumanez2011}.

For the operator $L$ with $\kappa\in\mathbb{C}(x)$, Kovacic's algorithm \cite{Kovacic1986} provides a complete effective procedure to determine whether $L(y)=0$ admits a non-zero Liouvillian solution and to compute it explicitly when it does. The algorithm is specific to $\mathbb{C}(x)$: it exploits partial fractions and the orders of poles (including the singularity at $\infty$) to decide the Galois case; no analogous procedure is available for a general differential field $F$. Thus, while the criterion of solvable $G^0$ holds in full generality over any $F$, algorithmic decidability requires $\kappa\in\mathbb{C}(x)$. The algorithm classifies the differential Galois group $G\subseteq\operatorname{SL}(2,\mathbb{C})$ into one of four mutually exclusive cases:
\begin{enumerate}[label=\textup{(\roman*)}]
\item\label{kov:reducible} \emph{Reducible.} $G$ is conjugate to a subgroup of the Borel group; the Riccati equation $w'+w^2+\kappa=0$ admits a solution in~$\mathbb{C}(x)$, and $y''+\kappa y=0$ has a solution of the form $e^{\int\omega}$ with $\omega\in\mathbb{C}(x)$.
\item\label{kov:imprimitive} \emph{Imprimitive.} $G$ is conjugate to a subgroup of the infinite dihedral group $D_\infty$; the Riccati equation has no solution in $\mathbb{C}(x)$ but admits one in a quadratic extension of~$\mathbb{C}(x)$.
\item\label{kov:finite} \emph{Primitive finite.} $G$ is a finite subgroup of $\operatorname{SL}(2,\mathbb{C})$ (a central extension of $A_4$, $S_4$, or $A_5$); all solutions of $y''+\kappa y=0$ are algebraic over~$\mathbb{C}(x)$.
\item\label{kov:full} \emph{Full.} $G=\operatorname{SL}(2,\mathbb{C})$; there are no Liouvillian solutions.
\end{enumerate}
In cases \ref{kov:reducible}--\ref{kov:finite}, the equation admits Liouvillian solutions; in case~\ref{kov:full}, it does not.

The following result formalizes the integrability consequences of the embedding $\Gamma\subset S$.

\begin{theorem}\label{thm:galois_integrability}
Consider the first-order ODE $u'=\phi(x,u)$ with $\mathcal{K}(x,u)=\kappa(x)$, and assume $\kappa$ belongs to a differential field $F$ with algebraically closed constant field~$\mathcal{C}$.
\begin{enumerate}[label=\textup{(\alph*)}]
\item\label{it:galois_forward} If $L(y)=0$ admits a non-zero Liouvillian solution over $F$, then the first-order ODE \eqref{eq:ode1} is integrable by quadratures.
\item\label{it:galois_converse} Conversely, if any two distinct solutions of the first-order ODE belong to a Liouvillian extension of $F$, then $L(y)=0$ admits a non-zero Liouvillian solution.
\end{enumerate}
In particular, when $\kappa\in\mathbb{C}(x)$, Kovacic's algorithm applied to $L(y)=0$ decides the Liouvillian integrability of the nonlinear first-order ODE.
\end{theorem}

\begin{proof}
For part~\ref{it:galois_forward}, suppose $\delta$ is a non-zero Liouvillian solution of $y''+\kappa(x)y=0$. By \Cref{thm:jacobi_integration}, $\delta$ can be used to obtain an integrating factor for the first-order ODE, and the general solution is obtained by quadratures.

For part~\ref{it:galois_converse}, let $u_0$ and $u_1$ be two distinct solutions of \eqref{eq:ode1} belonging to a Liouvillian extension of $F$. By \Cref{prop:nonhomog}, both satisfy $u''+\kappa(x)u=c(x)$. Therefore $y:=u_1-u_0$ is a non-zero element of the solution space $V$ of $y''+\kappa(x)y=0$, and it is Liouvillian since $u_0$ and $u_1$ are.
\end{proof}

\begin{remark}[Complexity floor]\label{rem:complexity_floor}
The embedding $\Gamma\subset S$ implies that the solutions of the nonlinear first-order ODE inherit a \emph{minimal} transcendental complexity dictated by $y''+\kappa(x)y=0$. In terms of Picard--Vessiot theory, any differential field extension containing two distinct solutions of \eqref{eq:ode1} must contain a non-zero solution of $y''+\kappa(x)y=0$ (by \Cref{thm:galois_integrability}\ref{it:galois_converse}). Consequently, if Kovacic's algorithm determines that $y''+\kappa(x)y=0$ has no Liouvillian solutions, then the general solution of the nonlinear first-order ODE is guaranteed to be non-Liouvillian. For instance, if $\kappa(x)=x$, the solutions are necessarily expressed in terms of Airy functions; no choice of the nonlinearity~$\phi$ can reduce them to elementary functions.
\end{remark}

\begin{remark}[No new transcendence]\label{rem:complexity_ceiling}
Conversely, the embedding constrains the solutions from above: every solution $u(x)$ of the first-order ODE belongs to the affine space $S=u_p+V$ and is therefore expressible in terms of the fundamental solutions of $y''+\kappa(x)y=0$ and the particular solution~$u_p$. In particular, the differential field generated by any single solution of \eqref{eq:ode1} is contained in the Picard--Vessiot extension for $y''+\kappa(x)y=0$ extended by~$c(x)$. The nonlinearity of $\phi$ thus acts as a selection rule, restricting $S$ to the one-dimensional locus $\Gamma$, but it cannot introduce differential-field complexity beyond that already present in the linear operator $L=\frac{d^2}{dx^2}+\kappa(x)$.
\end{remark}

\begin{remark}[Algorithmic decidability]\label{rem:decidability}
\Cref{thm:galois_integrability} renders the integrability of the class of first-order ODEs satisfying $\mathcal{K}(x,u)=\kappa(x)$ algorithmically decidable when $\kappa$ is rational: one applies Kovacic's algorithm to $\kappa$ and obtains, in finitely many steps, either an explicit Liouvillian solution of $y''+\kappa y=0$, and hence an integrating factor for \eqref{eq:ode1}, or a proof that no Liouvillian solution exists. This property is normally exclusive to linear equations; the geometric condition $\mathcal{K}(x,u)=\kappa(x)$ extends it to the present nonlinear class.
\end{remark}

\begin{remark}[Galois classification and projective geometry]\label{rem:galois_projective}
Assume $\kappa\in\mathbb{C}(x)$, so that Kovacic's classification applies.
The four cases can be interpreted within the projective framework of \Cref{geometric_interpretation}. The differential Galois group $G$ acts on the solution space $V$ of $y''+\kappa y=0$ and hence on $\mathbb{P}(V)$. The structure of the Galois orbits in $\mathbb{P}(V)$ determines the nature of the Riccati solutions:
\begin{itemize}
\item In case~\ref{kov:reducible}, $G$ fixes a point in $\mathbb{P}(V)$, corresponding to a Riccati solution in~$\mathbb{C}(x)$.
\item In case~\ref{kov:imprimitive}, $G$ permutes two points in $\mathbb{P}(V)$, corresponding to a pair of Riccati solutions exchanged by a quadratic automorphism.
\item In case~\ref{kov:finite}, $G$ is finite and acts on $\mathbb{P}(V)$ with finite orbits; all Riccati solutions are algebraic over~$\mathbb{C}(x)$.
\item In case~\ref{kov:full}, $G=\operatorname{SL}(2,\mathbb{C})$ acts transitively on $\mathbb{P}(V)$; there are no algebraic or Liouvillian Riccati solutions.
\end{itemize}
Via the tangent interpretation of \Cref{prop:tangent_interpretation}, the Galois group simultaneously constrains the tangent directions of the curve $\Gamma$ in the affine plane $S$, thereby governing the analytic complexity of the solutions of the nonlinear first-order ODE.
\end{remark}

We close this section with three examples illustrating the Galois-theoretic classification. They correspond, respectively, to Kovacic Cases~\ref{kov:reducible}, \ref{kov:imprimitive}, and~\ref{kov:full}.

\begin{example}[Continuing with \Cref{ex:euler_cauchy}: Kovacic Case~\ref{kov:reducible}]\label{ex:euler_cauchy_galois}
The Schr\"odinger equation \eqref{eq:euler_cauchy_2nd} has polynomial solutions (e.g., $y_1=x^2$), so the Riccati equation admits the rational solution $w=2/x$; this places it in Kovacic Case~\ref{kov:reducible} with a reducible Galois group. By \Cref{thm:galois_integrability}\ref{it:galois_forward}, the nonlinear ODE \eqref{eq:euler_cauchy_ode} is integrable by quadratures, as the explicit solution \eqref{eq:euler_cauchy_sol} confirms.
\end{example}

\begin{example}[Kovacic Case~\ref{kov:imprimitive}: an imprimitive Galois group]\label{ex:kovacic_ii}
Consider the nonlinear first-order ODE
\begin{equation}\label{eq:kovacic_ii_ode}
u' = u \left( \sqrt{x} \coth \left( \frac{4}{3}x^{3/2} \right) - \frac{1}{4x} \right) - \frac{\sqrt{x u^2 - 2\sqrt{x} \sinh \left( \frac{4}{3}x^{3/2} \right)}}{\sinh \left( \frac{4}{3}x^{3/2} \right)}
\end{equation}
defined on a suitable domain with $x>0$. A computation yields 
\[
\mathcal{K}(x,u) = -x - \frac{5}{16x^2} = \kappa(x).
\]
The associated Schr\"odinger equation is
\begin{equation}\label{eq:kovacic_ii_schrodinger}
y''(x) - \left( x + \frac{5}{16x^2} \right) y(x) = 0
\end{equation}
Applying Kovacic's algorithm to \eqref{eq:kovacic_ii_schrodinger}, one finds that Case~\ref{kov:reducible} fails (the Riccati equation admits no solution in $\mathbb{C}(x)$) but Case~\ref{kov:imprimitive} succeeds: the differential Galois group is conjugate to a subgroup of the infinite dihedral group $D_\infty\subset\operatorname{SL}(2,\mathbb{C})$. The Riccati equation admits a solution in a quadratic extension of $\mathbb{C}(x)$, and consequently \eqref{eq:kovacic_ii_schrodinger} possesses Liouvillian solutions. By \Cref{thm:galois_integrability}\ref{it:galois_forward}, the nonlinear ODE \eqref{eq:kovacic_ii_ode} is integrable by quadratures. In the projective framework of \Cref{rem:galois_projective}, the Galois group permutes two points in $\mathbb{P}(V)$, corresponding to a pair of Riccati solutions exchanged by a quadratic automorphism.
\end{example}

\begin{example}[Kovacic Case~\ref{kov:full}: Airy curvature and non-Liouvillian solutions]\label{ex:airy}
Consider any first-order ODE $u'=\phi(x,u)$ whose curvature satisfies $\mathcal{K}(x,u)=x$. The associated Schr\"odinger equation is the Airy equation
\begin{equation}\label{eq:airy_eq}
y''(x)+x\,y(x) = 0,
\end{equation}
whose differential Galois group over $\mathbb{C}(x)$ is $\operatorname{SL}(2,\mathbb{C})$: Kovacic's algorithm applied to \eqref{eq:airy_eq} determines that Cases~\ref{kov:reducible}--\ref{kov:finite} do not hold (Case~\ref{kov:full}; see, e.g., \cite{Kovacic1986}).

By \Cref{thm:galois_integrability}\ref{it:galois_converse}, no first-order ODE with curvature $\mathcal{K}(x,u)=x$ admits a Liouvillian general solution, regardless of the specific form of the nonlinearity~$\phi$. The solutions are necessarily expressed in terms of Airy functions; no choice of~$\phi$ can reduce them to elementary or Liouvillian expressions (cf.\ \Cref{rem:complexity_floor}).
\end{example}

\section{Conclusions}\label{sec:conclusions}

We have studied first-order ODEs $u'=\phi(x,u)$ satisfying the geometric condition $\mathcal{K}(x,u)=\kappa(x)$, i.e., the intrinsic Gauss curvature of the associated Riemannian surface depends only on the independent variable. The main results establish a threefold connection between this nonlinear class and the second-order linear operator $L=d^2/dx^2+\kappa(x)$:
\begin{enumerate}[label=\textup{(\roman*)}]
\item The divergence along every solution satisfies the Riccati equation $p'+p^2+\kappa=0$, which linearizes to the Schr\"odinger equation $L(y)=0$ via the classical substitution $p=y'/y$.
\item Every solution of the first-order ODE satisfies the non-homogeneous equation $L(u)=c(x)$, and the solution set $\Gamma$ is confined to the two-dimensional affine space $S$ determined by~$L$.
\item Solutions of $L(y)=0$ furnish integrating factors for the original nonlinear equation.
\end{enumerate}

This embedding $\Gamma\subset S$ has several consequences for integrability. It controls the analytic complexity of the nonlinear equation from both sides: the solutions can be no simpler and no more transcendental than those dictated by~$L$. The integrability by quadratures of the first-order ODE is equivalent to $L$ admitting a non-zero Liouvillian solution, and when $\kappa\in\mathbb{C}(x)$, Kovacic's algorithm provides a complete and effective decision procedure, a property normally exclusive to linear equations. Within the affine-space framework, the Riccati solutions arising from the curvature condition acquire a projective-geometric interpretation as tangent directions of the curve $\Gamma$, yielding an analogue of the Gauss map.

Several directions remain open. The affine space $S$ is determined by $\kappa$ alone, yet different nonlinearities $\phi$ sharing the same $\kappa$ trace different curves $\Gamma\subset S\cong\mathbb{R}^2$; the shape of $\Gamma$ therefore encodes information about $\phi$ beyond what $\kappa$ captures. A description of $\Gamma$ in terms of $\phi$ would yield a finer geometric invariant of the nonlinear ODE, intermediate between the coarse data of $\kappa$ and the full equation. A second direction concerns other conditions on the curvature operator $\mathcal{K}$. The present paper treats the case $\partial_u\mathcal{K}=0$; one may ask what integrability theory arises when $\mathcal{K}$ depends only on $u$, or satisfies some other algebraic or differential constraint. Such conditions would define new geometric classes of first-order ODEs, potentially admitting linearization procedures or Galois-theoretic treatments analogous to those developed here.

\section*{Declaration of generative AI and AI-assisted technologies in the manuscript preparation process}

During the preparation of this work the authors used Claude (Anthropic) in order to assist with the preparation and editing of the manuscript. After using this tool, the authors reviewed and edited the content as needed and take full responsibility for the content of the published article.

\section*{Declaration of competing interests}

The authors declare that they have no known competing financial interests or personal relationships that could have appeared to influence the work reported in this paper.

\section*{Funding}

This research did not receive any specific grant from funding agencies in the public, commercial, or not-for-profit sectors.

\bibliographystyle{elsarticle-num}
\bibliography{references}

\end{document}